\documentclass[a4paper,leqno]{amsart}

\usepackage{amsmath}
\usepackage{amsthm}
\usepackage{amssymb}
\usepackage{amsfonts}
\usepackage{mathrsfs}
\def\R{{\mathbb R}}% real numbers
\def\N{{\mathbb N}}% nonnegative integers
\def\le{\leqslant}
\def\ge{\geqslant}

\def\d{{\partial}}

\newcommand{\conj}[1]{\overline{#1}}

\newcommand{\pare}[1]{\left(#1\right)}

{\catcode `\@=11 \global\let\AddToReset=\@addtoreset}
\AddToReset{equation}{section}

\newcommand{\re}{\mathrm{Re}}
\newcommand{\im}{\mathrm{Im}}

\DeclareMathOperator{\diver}{div}

\theoremstyle{plain}
\newtheorem{theorem}{Theorem}[section]
\newtheorem{lemma}[theorem]{Lemma}
\newtheorem{corollary}[theorem]{Corollary}

\theoremstyle{definition}

\newtheorem{assumption}[theorem]{Assumption}

\newtheorem{remark}[theorem]{Remark}
\newtheorem*{remark*}{Remark}

\begin{document}

\title[Cauchy Problem for NLs with rotation term]
{On the Cauchy Problem for nonlinear Schr\"odinger equations with rotation}
\author[P.\ Antonelli]{Paolo Antonelli}
\address[P.\ Antonelli]{Department of Applied Mathematics and Theoretical
Physics\\
CMS, Wilberforce Road\\ Cambridge CB3 0WA\\ England}
\email{p.antonelli@damtp.cam.ac.uk}
\author[D.\ Marahrens]{Daniel Marahrens}
\address[D.\ Marahrens]{Department of Applied Mathematics and Theoretical
Physics\\
CMS, Wilberforce Road\\ Cambridge CB3 0WA\\ England}
\email{D.O.J.Marahrens@damtp.cam.ac.uk}
\author[C.\ Sparber]{Christof Sparber}
\address[C.\ Sparber]{Department of Mathematics, Statistics, and Computer Science\\
University of Illinois at Chicago\\
322 Science and Engineering Offices (M/C 249)\\
851 South Morgan Street\\
Chicago, Illinois 60607\\ United States}
\email{sparber@uic.edu}
\begin{abstract}
We consider the Cauchy problem for (energy-subcritical) nonlinear Schr\"odinger equations with sub-quadratic external potentials and an additional angular momentum rotation term. This equation is a well-known model for 
superfluid quantum gases in rotating traps. We prove global existence (in the energy space) for defocusing nonlinearities without 
any restriction on the rotation frequency, generalizing earlier results given in \cite{hai, HHL}. Moreover, we find that the rotation term has a considerable influence in proving finite time blow-up in the focusing case. 
\end{abstract}

\date{\today}


\thanks{This publication is based on work supported by Award No. KUK-I1-007-43,
funded by the King Abdullah University of Science and Technology (KAUST). C.\ S.\ acknowledges support by the Royal society through his University research fellowship. D.\ M.\ acknowledges support by the Cambridge European Trust and the EPSRC}
\maketitle

\section{Introduction}
Ever since the realization of Bose-Einstein condensation (BEC) in dilute atomic gases, much attention has been given to dynamical phenomena associated to its superfluid nature.
One remarkable feature of a superfluid is the appearance of quantized vortices, cf. \cite{Af} for a broad introduction to these kind of phenomena. In physical experiments, the BEC is thereby set 
into rotation by a stirring potential, which is usually induced by a laser \cite{MCWD, MCBD, MAHHWC, SBBHD} (see also \cite{BaoDuZh} for numerical simulations). 
The corresponding mathematical model is a nonlinear Schr\"odinger equation (NLS) 
with angular momentum rotation term, i.e.
\begin{equation}\label{eq:nls_rotat}
i \hbar \d_t\psi=-\frac{\hbar^2}{2}\Delta\psi+ \lambda |\psi|^{2}\psi+V(x)\psi-\Omega\cdot L\psi,\quad (t, x)\in\R\times\R^3,
\end{equation}
where $\psi=\psi(t, x)$ is the complex-valued wave function of the condensate. In the physics literature, \eqref{eq:nls_rotat} is known as the \emph{Gross-Pitaevskii equation} for rotating Bose gases. 
The coupling constant $\lambda \in \R $ 
can be experimentally tuned to account for both \emph{defocusing} ($\lambda>0$) and \emph{focusing} ($\lambda<0$) nonlinearities. 
The potential $V(x)$ describes the 
magnetic trap and is usually assumed to be of the form
\begin{equation}
\label{eq:Vquadr}
V(x)=\frac{1}{2}\sum_{j=1}^3 \gamma_j^2 x_j^2, \quad \gamma_j \in \R.
\end{equation}
Finally, $\Omega\cdot L$ denotes rotation term, where
\begin{equation}
\label{eq:angular_momentum}
L:=-ix\wedge\nabla
\end{equation}
is the quantum mechanical \emph{angular momentum operator} and $\Omega\in\R^3$ is a given \emph{angular velocity vector}. For a rigorous derivation of \eqref{eq:nls_rotat} 
in the stationary case, we refer to \cite{LS}. Furthermore, we remark that the appearance of quantized vortices has been rigorously proved in \cite{seir}, 
by means of a spontaneous symmetry breaking for the ground state of the stationary equations 
(provided $\lambda>0$ is sufficiently big). For further mathematical results in this direction we refer to \cite{Af} and the references given therein.

The aforementioned works illustrate the fact that there is a considerable amount of mathematical studies devoted the stationary equation. On the other hand, the 
time-dependent equation \eqref{eq:nls_rotat}, has not been given as much attention, even though it is considered to provide the basis for a 
dynamical description of vortex creation. Indeed, except for numerical simulations \cite{BaoDuZh} , we are only aware of \cite{hai, HHL} providing rigorous results for \eqref{eq:nls_rotat}. 
In \cite{HHL} global well-posedness of the Cauchy problem 
(in the energy space) is proved in the case where $\lambda>0$, $V(x)=\frac{\gamma^2}{2}|x|^2$, i.e. an isotropic confinement, and $|\Omega|=\gamma$. The analogous result for in $d=2$ spatial dimensions is given in \cite{hai}.

\begin{remark} Let us mention that in \cite{Liu, LiuSp}, the NLS model \eqref{eq:nls_rotat} is also rigorously studied. The results, however, mainly concern an asymptotic regime, the so-called \emph{semi-classical limit}, and 
are thus very different from the present work.
\end{remark}

In view of these results the main goal of our work is twofold: First, we shall prove global well-posedness of \eqref{eq:nls_rotat} in the defocusing case, 
without any restriction on $|\Omega|$ or $\{\gamma_j\}_{j=1}^3$. The latter is needed to describe actual physical experiments, which often require 
$|\Omega|\not =\gamma$. To this end, we shall show that by a suitable time-dependent change of coordinates, we can transform equation \eqref{eq:nls_rotat} into a nonlinear 
Schr\"odinger equation \emph{without} rotation term 
but with a \emph{time-dependent trapping potential}. In a second step, we shall analyze the possibility of finite time blow-up of solutions, in the case of a focusing nonlinearity.  
Recall that finite time blow-up means, that there is a $T^* < +\infty$, such that
\[
\lim_{t \to T^*} \| \nabla \psi (t) \|_{L^2} = +\infty.
\]
%This notion is correct also here, see Lemma~\ref{lem:locEx}. 
As we shall see, the usual proof of finite time blow-up, based on the classical virial argument of Glassey (see e.g.\ \cite{caze}), in general does not go through in a straight-forward way, due to the influence of the rotation term. 
Instead, it has to be slightly modified, yielding blow-up conditions which depend on $|\Omega|$, and which coincide with the usual conditions in the limit $|\Omega|\rightarrow 0$.

\section{Mathematical Setting and Main Result}
\label{sec:intro}

In the following we shall consider the Cauchy problem for the following, slightly more general NLS type model
\begin{equation}
\label{eq:NLS_rot}
i \partial_t \psi = -\frac{1}{2} \Delta \psi + \lambda |\psi|^{2\sigma} \psi + V(x) \psi - \Omega \cdot L \psi,\quad 
\psi(0)=\psi_0(x),
\end{equation}
where $x\in \R^d$, for $d= 2$ or $d=3$, respectively, and $\sigma < \frac{2}{d-2}$, i.e. the nonlinearity is assumed to be \emph{energy-subcritical}.
In $d=2$ the rotation term simply reads
\begin{equation}
\label{eq:2d}
\Omega \cdot L = - i |\Omega| (x_1 \partial_{x_2} - x_2 \partial_{x_1}).
\end{equation}
\begin{remark} In $d=3$ we could, without restriction of generality, choose a reference frame such that $\Omega = (0,0,|\Omega |)^\top$, yielding the same formula as in \eqref{eq:2d}. For the sake of generality
we shall not do so but consider the term $\Omega \cdot L \equiv -i \Omega\cdot (x\wedge \nabla)$ in full generality.
\end{remark}
A potential $V(x) \in \R^d$, satisfying $(\Omega\cdot L) V (x)= 0, \quad \forall x \in \R^d$, is said to be \emph{axially symmetric} (with respect to the rotation axis $\Omega \in \R^3$). 
In particular, this holds in the case of 
an isotropic trap potential, i.e.\ a potential of the form \eqref{eq:Vquadr} with $\gamma_1=\gamma_2=\gamma_3$.

Formally, \eqref{eq:NLS_rot} preserves the total mass
\begin{equation*}
M:=\int_{\R^d}|\psi(t,x)|^2 \, dx,
\end{equation*}
and the energy 
\begin{equation}\label{eq:en}
E_\Omega:=\int_{\R^d}\frac{1}{2}|\nabla\psi|^2+V(x)|\psi|^2+\frac{\lambda}{\sigma+1}|\psi|^{2\sigma+2} - \conj{\psi} (\Omega \cdot L) \psi \, dx.
\end{equation}
Note that, the last term is indeed real valued (as can be seen by a partial integration).
In order for these two quantities to be well defined, we shall study the Cauchy problem corresponding to \eqref{eq:NLS_rot} in the space
\begin{equation*}
\Sigma:=\{f\in H^1(\R^d)\; : \; |x|f\in L^2(\R^d)\},
\end{equation*}
equipped with the norm
\[
\| f \|_{\Sigma}^2:= \| f \|_{L^2}^2 + \| \nabla f \|_{L^2}^2 + \| x f \|_{L^2}^2.
\]
We remark that even if the potential $V(x)$ is chosen to be identically zero, it would not be enough to consider the Cauchy problem for $\psi \in H^1(\R^d)$, since in this case we can no longer guarantee that
\begin{equation}\label{eq:ang}
L_\Omega:=\int_{\R^d} \overline{\psi} (\Omega \cdot L) \psi \,dx < +\infty.
\end{equation}
Physically speaking, this means that $\psi$ has finite angular momentum. The choice of $\Sigma$ is therefore natural in our situation and not necessarily linked to the presence of a harmonic trapping potential, 
in contrast to \cite{Car03,Car05, Car09}.

We can now state the main result of this work.
\begin{theorem}
\label{thm:main}
Let $0<\sigma  < 2/(d-2)$, $\lambda \in \R$, $\Omega \in \R^d$, for $d=2,3$ and denote the smallest trap frequency by $\underline \gamma:= \min \{\gamma_j \}_{j=1}^d$.\\
\emph{(1)} Then, for any given initial data $\psi_0\in \Sigma$, there exists a unique global in-time solution $\psi\in C([0,\infty);\Sigma)$ to \eqref{eq:NLS_rot}, provided one of the following conditions is satisfied: 
\begin{itemize}
\item[(i)]  the nonlinearity is $L^2$-subcritical $\sigma < 2/d$, or 
\item[(ii)] $\sigma \ge 2/d$ and $\lambda \ge 0$, i.e. the nonlinearity is defocusing.
\end{itemize}
\emph{{(2)}} On the other hand, if $\lambda < 0$, and if either:
\begin{itemize}
\item[(i)] $(\Omega\cdot L)V=0$, i.e. $V$ is axially symmetric, and $\sigma \ge 2/d$, or
\item[(ii)] $(\Omega\cdot L)V\neq0$, $|\Omega|\le \underline \gamma$, and $ \sigma\ge \alpha_\Omega/d $, where 
\begin{equation}\label{def:alph}\alpha_\Omega:=\sqrt{\frac{4\underline \gamma^2}{\underline \gamma^2-|\Omega|^2}}.
\end{equation}
\end{itemize}
Then there exist initial data $\psi_0 \in \Sigma$ such that finite time blow-up of the corresponding solution $\psi(t)$ occurs.
\end{theorem}

In fact, we shall prove Assertions (1)(i) and (ii) under the more general assumptions on $V(x)$, see Assumption \ref{ass:pot} below. This, together with the fact that no condition on the size of $|\Omega| $ or $\{\gamma_j\}_{j=1}^d$ 
is required, generalizes the earlier results given in \cite{hai, HHL}. 

Concerning the possibility of finite time blow-up, we see that one has to distinguish between the case of axially symmetric potential and the case where this symmetry is broken. The reason will become clear 
in the proof given below. In the case of a non-axially symmetric potential we can rigorously prove the occurrence of blow-up only in under the additional restrictions $|\Omega|\le \underline \gamma$, and 
$ \sigma\ge \alpha_\Omega/d $, i.e. only for a limited range of nonlinearities. It is easily seen that in $d=3$, the set of $\sigma$'s satisfying our conditions is non-empty, provided $|\Omega|^2 < \frac{8}{9}\underline \gamma^2$. 
Also note that in the case of vanishing rotation $\lim_{|\Omega|\to 0}\alpha_\Omega = 2$, yielding the usual range of $L^2$-supercritical nonlinearities. 
At this point it is not clear if these additional restrictions are only due to the strategy of our proof, or if they indicate an actual difference in the behavior of solutions to \eqref{eq:NLS_rot}. In particular, the question whether or not 
finite time blow-up occurs in situations where $\Omega >  \underline \gamma$ is completely open so far. In terms of physics, the latter would correspond to the case where the rotation is stronger than the trap and thus one would expect 
a behavior which is similar (at least qualitatively) to the ``free'' case, i.e. without any potential.

This paper is now organized as follows: Section \ref{sec:exist} is devoted to the proof of Assertion (1) of Theorem \ref{thm:main}. To this end, we shall first prove local in-time existence for solutions to 
\eqref{eq:NLS_rot}. Also, we shall see that a naive use of the conservation 
laws for mass and energy in general leads to restrictions on $|\Omega|$ or $\{\gamma_j\}_{j=1}^d$. We shall show in a second step how to overcome this problems using a coordinate-change. Assertion (2) of our main Theorem is then proved in Section \ref{sec:blow-up} and we finally 
collect some concluding remarks in Section \ref{sec:final}.

%%%%%%%%%%%%%%%%%%%%%%%%%%%%%%%%%%%%%%%%%%%%%%%%%%%%%%%

\section{Local and global existence}
\label{sec:exist}

In this section we shall allow for more general class of potentials $V(x)$ satisfying the following assumption.

\begin{assumption}
\label{ass:pot}
The potential $V:\R^d\to\R$ is assumed to be smooth and \emph{sub-quadratic}, i.e.\ for all multi-indices $k\in\N^d$, with $|k|\ge 2$, there exists a constant $C=C(k)>0$ such that
\begin{equation}\label{eq:pot}
|\d^k V(x)|\le C\qquad\textrm{for all }x\in\R^d.
\end{equation}
\end{assumption}
\begin{remark}Clearly, a harmonic trapping potential of the form \eqref{eq:Vquadr} is sub-quadratic. Assumption \ref{ass:pot} allows us to take into account more general situations of physical interest, such as a combined 
harmonic trap plus optical lattice potential, see e.g.\ \cite{CN}.
Note however, that under Assumption \ref{ass:pot}, the potential is not necessarily bounded below (or confining). In particular we can also allow for repulsive potentials such as $V(x) = - \gamma^2 |x|^2$, 
see \cite{Car03}.\end{remark}

As a first, preliminary step, we shall prove the following local well-posedness result.
\begin{lemma}
\label{lem:locEx}
Let $\psi_0\in\Sigma$, $\omega \in \R$, and $0< \sigma  < 2/(d-2)$. Moreover, assume that $V$ satisfies Assumption~\ref{ass:pot}. 
Then there exists a time $T=T(\|\psi_0\|_\Sigma)>0$ and a unique maximal solution $\psi\in C([0,T);\Sigma)$ of equation~\eqref{eq:nls_rotat} with $\psi(0)=\psi_0$. The solution is maximal in the sense 
that, if $T<+\infty$, then
\[
 \lim_{t\rightarrow T} \|\nabla \psi(t)\|_{L^2} = + \infty.
\]
Moreover, the following 
conservation laws hold
\begin{align}\label{eq:en_cons}
M(t) = M(0), \quad E_\Omega(t)&=E_\Omega(0),
\end{align}
whereas for the angular momentum we have
\begin{equation}\label{eq:ang_mom_cons}
L_\Omega(t)+\int_0^t\int_{\R^d} i |\psi|^2 (\Omega \cdot L) V(x)dx=L_\Omega(0).
\end{equation}
\end{lemma}
\begin{proof}
The proof is an adaptation of classical arguments, based on a contraction mapping (via Duhamel's formula) and Strichartz estimates 
for the linear (unitary) group $U(t)=e^{i t H}$ generated by the Hamiltonian
\[
H = -\frac{1}{2} \Delta + V(x) - \Omega \cdot L.
\]
In our case, Strichartz estimates can be obtained by following the approach of Kitada \cite{kita}. All we need to do is to check the assumptions given there: 
First, we note that the Hamiltonian is a quantization of the classical phase space function $H(x,\xi) = - |\xi|^2/2 + V(x) +i \Omega \cdot (x \wedge \xi)$, which is smooth and sub-quadratic in $x$ and $\xi $, due to Assumption \ref{ass:pot}. 
We easily check that also the second assumption of \cite{kita} holds true, i.e.
\[
 \langle H(x,\xi), e^{ix\cdot\xi} \conj{\phi(x)} \hat{\varphi}(\xi) \rangle = \langle H(x,\eta), e^{-ix\cdot\eta} \varphi(x) \conj{\hat{\phi}(\eta)} \rangle,
\]
where $\langle \cdot, \cdot \rangle$ denotes the usual duality bracket between tempered distributions and Schwartz functions. This expression is equivalent to the fact that $H$ is essentially self-adjoint. 
We therefore conclude that there is a $\delta >0$ such that
$$\| U(t) \varphi \|_{L^\infty} \le  \frac{1}{|t|^{d/2}} \| \varphi \|_{L^2}, \quad \mbox{for $|t| < \delta$.}$$ 
In particular it follows that the Strichartz estimates for $H$ are analogous to those found in the well-known case of 
NLS with quadratic potentials \cite{Car05}, i.e.\ the rotation term does not influence the dispersive behavior (locally in time). The existence of a local in-time solution then follows analogously as given therein. 
The conservation laws \eqref{eq:en_cons} follow from straightforward calculations in combination with 
a standard density argument (see e.g.\ \cite{caze}).
Finally, in order to prove the blow-up alternative we first compute
\begin{equation}
 \label{eq:first_moment_deriv}
 \frac{d}{dt} \|x\psi(t)\|_{L^2}^2 = 2 \im \int_{\R^d} x \conj{\psi} (t) \nabla \psi(t) \;dx \le \|x\psi(t)\|_{L^2}^2 + \|\nabla\psi(t)\|_{L^2}^2.
\end{equation}
Thus, as long as $\|\nabla\psi(t)\|_{L^2}$ is bounded, Gronwall's inequality yields a bound on $\|x\psi(t)\|_{L^2}$ as well. In view of mass conservation, the only obstruction to global existence is therefore given by the 
possible unboundedness of $\|\nabla\psi(t)\|_{L^2}$ in $[0,T]$. 
\end{proof}

\begin{remark}
For quadratic potentials of the form \eqref{eq:Vquadr}, Strichartz estimates can be obtained explicitly by invoking a generalization of Mehler's formula for the kernel of $U(t)$, c.f.\ \cite{Car09}. 
Indeed, by making the following ansatz 
\[
 U(t)\psi_0(x) = \prod_{j=1}^d \left(2\pi i \mu_j(t)\right)^{-1/2} \int_\R^d e^{\frac{i}{2}F(t,x,y)} \psi_0(y) \;dy
\]
where $\mu_j(t) \in \R_+$ and 
$F(t,x,y)$ is a general quadratic form in $x$ and $y$ with (yet to be determined) time-dependent coefficients. Substituting this into the linear Schr\"odinger equation yields a coupled system of differential equations for these coefficients. 
Solving this system, however, is in general rather tedious. This approach is therefore only feasible under some simplifying assumptions, such as $\Omega = 0$ \cite{Car03, Car05}, or 
$V(x)=\frac{\gamma^2}{2} |x|^2$ with $|\Omega| = \gamma$ as it is done in 
\cite{hai, HHL}.
\end{remark}

In view, of \eqref{eq:ang_mom_cons}, we immediately conclude the following important corollary.
\begin{corollary}\label{cor:tilde_en} 
If $V(x)$ is such that $(\Omega\cdot L) V = 0$, then we also have conservation of angular momentum, i.e. 
$L_\Omega(t) = L_\Omega(0)$, and in addition it holds
\begin{equation}\label{eq:tilde_en}
E_0(t)\equiv \int_{\R^d}\frac{1}{2}|\nabla\psi|^2+V(x)|\psi|^2+\frac{\lambda}{\sigma+1}|\psi|^{2\sigma+2}dx = E_0(0).
\end{equation}
\end{corollary}
Thus, in the  case of axially symmetric potentials $V(x)$, there are in fact two conserved energy functionals corresponding to \eqref{eq:NLS_rot}. 
\medskip

With a local existence result in hand, we can ask about global existence. In order to infer $T=+\infty$, one usually invokes the conservation of mass and energy \eqref{eq:en_cons}. 
The problem is, that due to the appearance of the angular momentum rotation term, the energy $E_\Omega(t)$ has no definite sign 
even if $V\ge0$ and $\lambda \ge 0$ (defocusing nonlinearity). 
A possible strategy to overcome this problem is to rewrite the linear Hamiltonian as
\begin{equation}
\label{eq:magn_op}
H=-\frac{1}{2}\Delta - \Omega \cdot L + V(x) = \frac{1}{2} (-i \nabla - A(x))^2 + V(x) - \frac{|\Omega|^2}{2}r^2,
\end{equation}
where $A (x)= \Omega\wedge x$ and $r= |x \wedge \Omega| / |\Omega|$ denotes the radial distance perpenticular to $\Omega$. Note that 
$A(x)$ can be considered as the vector potential corresponding to a constant magnetic field $B = \nabla \wedge A = 2\Omega$. 
The corresponding ``magnetic derivative'' $D_A:= - i (\nabla +A(x))$ 
is known to satisfy, cf. \cite[Chapter 7]{caze}:
\[
\| \nabla |\psi| \|_{L^2} \le \|D_A \psi\|_{L^2} \le \| \nabla \psi \|_{L^2} + \|x \psi\|_{L^2}.
\]
It can therefore be used to control the nonlinear potential energy $\propto \| \psi \|_{L^{2\sigma +2}}$ 
via Gagliardo-Nirenberg type inequalities. If in addition, $V(x)$ is given by \eqref{eq:Vquadr} with $|\Omega| \le \underline \gamma$ we infer that 
$ V(x) - \frac{|\Omega|^2}{2}r^2 \ge 0$. In this case, the linear part of the energy is seen to be a sum of non-negative terms, and global existence can be concluded 
as in the case of NLS with quadratic confinement \cite{Car05}.
However, it seems impossible to extend this approach to situations in which $|\Omega| > \gamma$, even if $\lambda >0$. 
In order to do so, we shall follow a different idea, which invokes a time-dependent change of coordinates.

\begin{proof}[Proof of Assertion (1) of Theorem~\ref{thm:main}]

We start with the $L^2$-subcritical case, i.e. $0< \sigma < 2/d$ which follows by standard arguments. Namely, we write the nonlinear Schr\"odinger equation as the solution of a fixed point equation, using Duhamel's formula:
\begin{align*}
\psi(t) &= U(t)\psi_0 - i \lambda \int_0^t U(t-s)\left(|\psi|^{2\sigma}(s)\psi(s)\right) \;ds\\ &=:\Phi(\psi)(t).
\end{align*}
Next, we calculate the commutators $[\nabla, H]$ and $[x, H]$. Explicitly, we find
\begin{align*}
 [\nabla, H] =& -\frac{1}{2}[\nabla, \Delta] + [\nabla, V] + i[\nabla,\Omega\cdot(x\wedge\nabla)]\\
=& \, \nabla V + i [\nabla, x\cdot(\nabla\wedge\Omega)]\\
= & \, \nabla V + i \nabla\wedge\Omega
\end{align*}
by the well-known formula $a\cdot(b\wedge c) = \det(a,b,c) = (a\wedge b)\cdot c$ for three-dimensional vectors $a,b,c$. Similar calculations yield
\[
 [x, H] = \nabla - i \Omega\wedge x
\]
From this we can deduce
\begin{align*}
\nabla \Phi(\psi)(t) = & \ U(t) \nabla \psi_0 - i \lambda \int_0^t U(t-\tau) \nabla \left(|\psi(\tau)|^{2\sigma}\psi(\tau)\right) \;d\tau  \\ & - i \int_0^t U(t-\tau)\left( \nabla V - i \Omega \wedge \nabla \right) \Phi(\psi)(\tau) \;d\tau,
\end{align*}
and
\begin{align*}
x \Phi(\psi)(t) = & \ U(t) \nabla \psi_0 - i \lambda \int_0^t U(t-\tau) \nabla \left(|\psi(\tau)|^{2\sigma}\psi(\tau)\right) \;d\tau \\ &  - i \int_0^t U(t-\tau)\left( \nabla - i \Omega\wedge x \right) \Phi(\psi)(\tau) \;d\tau.
\end{align*}
Using Strichartz estimates (see the discussion in the proof of Lemma~\ref{lem:locEx}) we have that
\begin{equation*}
\begin{split}
\|\psi\|_{L^p(0,T;L^q)\cap L^\infty(0,T;L^2)} \le C \|\psi_0\|_{L^2} + C \|\psi\|_{L^k(0,T;L^q)}^{2\sigma}\,,
\end{split}
\end{equation*}
where
\[
q = 2\sigma+2, \quad p = \frac{4\sigma+4}{d\sigma}, \quad k = \frac{2\sigma(2\sigma+2)}{2-(d-2)\sigma}.
\]
If $\sigma<2/d$ it holds that $1/p < 1/k$ and thus
\[
\|\psi\|_{L^k(0,T;L^q)} \le T^{1/k-1/p} \|\psi\|_{L^p(0,T;L^q)}.
\]
If we choose $T=T^*>0$ small enough, 
we can absorb the last term on the right hand side in order to get a bound on $\|\psi\|_{L^p(0,T^*;L^q)\cap L^\infty(0,T^*;L^2)}$. Since we can shift the time interval $[0,T^*]$ by an 
arbitrary amount of time, in the same way we can get a uniform bound on $\|\psi\|_{L^p(I;L^q)\cap L^\infty(I;L^2)}$ for every interval of length $|I|\le T^*$.
Thus, by splitting any arbitrarily large time interval $[0,T]$ into sufficiently small sub-intervals $\{ I_n \}_{n =1}^N$ such that $|I_n|\le T^*$ and iterating the bound 
$\|\psi\|_{L^p(I;L^q)\cap L^\infty(I;L^2)}\le C^*$, we infer  $\|\psi\|_{L^p(0,T;L^q)\cap L^\infty(0,T;L^2)}\le C$ where $C<+\infty$ depends on the value of $T$. 
Moreover, since $\psi_0\in\Sigma$ we also get
\begin{align*}
& \|x\psi\|_{L^p(0,T;L^q)\cap L^\infty(0,T;L^2)} + \|\nabla\psi\|_{L^p(0,T;L^q)\cap L^\infty(0,T;L^2)}  \\
&\le  C \|\psi_0\|_{\Sigma}   + C \|\psi\|_{L^k(0,T;L^q)}^{2\sigma} \left(\|x\psi\|_{L^p(0,T;L^q)} + \|\nabla\psi\|_{L^p(0,T;L^q)}\right) + \\
& + C T \left( \|x\psi\|_{L^\infty(0,T;L^2)} + \|\nabla\psi\|_{L^\infty(0,T;L^2)} \right)
\end{align*}
From here, we proceed as before to obtain a uniform bound for the left hand side for small $T^*$ and thus by iteration for arbitrary time intervals $[0,T]$.

Next we consider the case of an $L^2$-supercritical nonlinearity $\sigma> 2/d$. In this case, the iterative argument given above breaks down. 
The basic idea is to use a change of coordinates in order to bring equation (\ref{eq:NLS_rot}) into 
a more suitable form. For the sake of notation we shall only consider the case $d=3$ in the following.
We first note that by using the skew-symmetric matrix
\[
 \Theta :=
\begin{pmatrix}
  0 & \Omega_3 & -\Omega_2 \\
  -\Omega_3 & 0 & \Omega_1 \\
  \Omega_2 & -\Omega_1 & 0
\end{pmatrix},
\]
the wedge product with the angular momentum can be written as
\[
 \Omega \wedge x = -\Theta\cdot x.
\]
Then the matrix-exponential
\[
 X(t,x):={\rm e}^{\Theta t}\cdot x
\]
defines a rotation of the vector $x\in\R^3$ around the axis $\Omega$ by an angle of $-|\Omega|t$. Its time-derivative can be calculated as
\begin{equation}
 \label{eq:partial_rot}
 \partial_t X(t,x) = \Theta \cdot X(t,x) = -\Omega \wedge X(t,x).
\end{equation}
Denoting the wave function in rotated coordinates via
\[
 \widetilde{\psi}(t,x) = \psi(t, X(t,x)),
\]
we conclude from \eqref{eq:partial_rot} that
\[
 i\partial_t \widetilde{\psi}(t,x) = i\partial_t \psi(t,X(t,x)) - i\left(\Omega\wedge X(t,x)\right) \cdot \nabla \psi(t, X(t,x)).
\]
Rewriting $-i(\Omega\wedge X)\cdot \nabla = -i\Omega\cdot(X\wedge\nabla) = \Omega\cdot L$, we arrive at 
\[
 i\partial_t \widetilde{\psi}= -\frac{1}{2}\Delta\widetilde{\psi}+ \lambda |\widetilde{\psi}|^{2\sigma}\widetilde{\psi} + V(X(t,x))\widetilde{\psi}.
\]
where we have also used the fact that the Laplace operator is invariant with respect to rotations, i.e.
\[
 \Delta_{X} \psi(t,X(t,x)) = \Delta_x \psi(t,X(t,x)).
\]
Dropping all the tildes and denoting $W(t,x) = V(X(t,x))$, we conclude that up to a change of coordinates, equation~\eqref{eq:NLS_rot} is equivalent to the following NLS with time-dependent potential
\begin{equation}\label{eq:NLS_time_dept_pot}
i \partial_t \psi = -\frac{1}{2} \Delta \psi + \lambda |\psi|^{2\sigma} \psi + W(t,x) \psi
\end{equation}
Note that $W(t,x)$ is smooth w.r.t. $t\in \R$ and sub-quadratic w.r.t. $x\in \R^3$ with the same (uniform) constants $C(k)$ as given in Assumption~\ref{ass:pot} for $V(x)$.
Moreover, if $V(x)$ is axially symmetric, i.e. $(\Omega \cdot L) V (x) = 0$, equation \eqref{eq:partial_rot} implies that 
\[
 \partial_t W(t,x) = - \Omega\wedge X(t,x)\cdot\nabla V(X(t,x)) = -i(\Omega \cdot L) V(X(t,x))  = 0,
\]
and hence $W(t,x)= W(0,x) \equiv V(x)$. The energy corresponding to the transformed NLS \eqref{eq:NLS_time_dept_pot} is given by 
\[
E_W(t) := \int \frac{1}{2} |\nabla \psi(t,x)|^2 + \lambda |\psi(t,x)|^{2\sigma+2} + W(t,x)|\psi(t,x)|^2 \;dx.
\]
However, since the potential $W(t,x)$ in general is time-dependent, the $E_W(t)$ is no longer a conserved quantity. 
Rather, we obtain that
\begin{equation}
 \label{eq:energy_deriv}
\frac{d}{dt} E_W(t) = \int \partial_t W(t,x) |\psi(t,x)|^2 \;dx.
\end{equation}
Nevertheless it is not hard to prove Assertion (1)(ii) of Theorem~\ref{thm:main}: 
In view of the blow-up alternative, stated in Lemma \ref{lem:locEx}, it suffices to show $\| \nabla\psi(t) \|_{L^2} < + \infty$, for all $T>0$. To this end, we first estimate 
\[
 \frac{1}{2}\|\nabla\psi(t)\|_{L^2}^2 \le E_W(t) + \left|\int W(t,x) |\psi(t,x)|^2 \;dx\right| \le E_W(t) + C \|x\psi(t)\|_{L^2}^2,
\]
under the assumption that $\lambda >0$. 
Integrating equation~\eqref{eq:energy_deriv} and having in mind that $W(t,x)$ is sub-quadratic in $x$, we obtain that
\begin{align}\label{eq:inequ}
\|\nabla\psi(t)\|_{L^2}^2 \le& \, E_W(0) + \int_0^t \frac{d}{ds}{E_W}(s) \;ds + C \|x\psi(t)\|_{L^2}^2\\
\le& \, C_0 \left( 1 + \|x\psi(t)\|_{L^2}^2 + \int_0^t \|x\psi(s)\|_{L^2}^2 \;ds \right).\notag
\end{align}
Recalling inequality \eqref{eq:first_moment_deriv}, we infer
\[
 \frac{d}{dt} \|x\psi(t)\|_{L^2}^2 + \|x\psi(t)\|_{L^2}^2 \le C_0 \left( 1 + \|x\psi(t)\|_{L^2}^2 + \int_0^t \|x\psi(s)\|_{L^2}^2 \;ds \right)
\]
which by Gronwall's inequality yields an uniform bound on $\|x\psi(t)\|_{L^2}$ for every time interval $[0,T]$.  With this in hand, we can bound $\|\nabla\psi(t)\|_{L^2}$ by simply using inequality~\eqref{eq:inequ} once more.
\end{proof}

\begin{remark}
In particular, for $\Omega = (0,0,|\Omega| )^\top$ and $V(x)$ given by \eqref{eq:Vquadr}, we explicitly find 
\begin{align*}
W(t,x) = \frac{1}{2}\Big (\left( \gamma_1^2 \cos^2(|\Omega| t) + \gamma_2^2 \sin^2(|\Omega| t) \right) x_1^2 + \left( \gamma_1^2 \sin^2(|\Omega| t) + \gamma_2^2 \cos^2(|\Omega| t) \right) x_2^2 \\+ \sin(2 |\Omega| t) \left( \gamma_1^2 - \gamma_2^2 \right) x_1 x_2 + \gamma_3^2 x_3^2\Big).
\end{align*}
Clearly, $W= \frac{1}{2} (\gamma_1^2 x_1^2+ \gamma_2^2 x_2^2+\gamma_3^2 x_3^2)$ in the axially symmetric case $\gamma_1^2 = \gamma_2^2$.
\end{remark}

\section{Finite time blow-up}\label{sec:blow-up}

This section is devoted to the proof of assertion (2) of Theorem~\ref{thm:main}. It follows from the following lemma.
\begin{lemma}\label{prop:blowup}
Let  $\lambda <0 $, $ \sigma < 2/(d-2)$, $\Omega \in \R^d$, for $d=2,3$, and $V(x)$ be a quadratic potential of the form \eqref{eq:Vquadr}. Denote $\underline \gamma =\min \{ \gamma_j \} _{j=1}^d$ and let $\alpha_\Omega$ be as in \eqref{def:alph}. If either
\begin{itemize}
\item[(i)] $(\Omega\cdot L)V=0$, $\sigma \ge 2/d$, and $E_0(0) <0$, or
\item[(ii)] $(\Omega\cdot L)V\neq0$, $|\Omega|\le \underline \gamma$, $ \sigma\ge \alpha_\Omega/d $, and $E_\Omega(0)<0$, 
\end{itemize}
then the corresponding solution necessarily blows up in finite time.
\end{lemma}
Note that the condition for the energy of the initial data $\psi_0$ are not identical in both cases. The reason will become clear in the proof given below. 

\begin{proof} To simplify the arguments later on, let us first compute the conservation laws for the mass and momentum densities, i.e. $\rho:=|\psi|^2$ and $J:=\im(\conj{\psi}\nabla\psi)$. Indeed a straightforward calculation yields
\begin{equation}\label{eq:mass}
\d_t\rho+\diver J=i\Omega \cdot L \rho.
\end{equation}
On the other hand, for the current density $J$ we find
\begin{equation}\label{id:current}
\begin{aligned}
\d_t\pare{\im(\conj{\psi}\nabla\psi)}=&\ \im\pare{\pare{-\frac{i}{2}\Delta\conj{\psi}+iV(x)\conj{\psi}+i\lambda|\psi|^2\psi
+i\Omega\cdot L\conj{\psi}}\nabla\psi}\\
&+\im\pare{\conj{\psi}\nabla\pare{\frac{i}{2}\Delta\psi-iV(x)\psi-i\lambda|\psi|^{p-1}\psi+i(\Omega\cdot L\psi)}}
\end{aligned}
\end{equation}
Next, we calculate
\begin{equation*}
\im\pare{\conj{\psi}\nabla(i\Omega\cdot L\psi)}= \im\pare{\conj{\psi}(i\Omega\cdot L)\nabla\psi}-\Omega\wedge J.
\end{equation*}
Thus we can combine the two terms in \eqref{id:current} which stem from the rotation via
\[
\im\pare{(i\Omega\cdot L)\conj{\psi}\nabla\psi}+\im\pare{\conj{\psi}(i\Omega\cdot L)\nabla\psi}-\Omega\wedge J=(i\Omega\cdot L)J-\Omega\wedge J
\]
where we have used that $i\Omega\cdot L$ is real-valued. The other terms in \eqref{id:current} are usual in quantum hydrodynamics, see e.g.\ \cite{paolo}, yielding the following equation for $J$:
\begin{equation}\label{eq:current}
\d_tJ+\diver\pare{\re(\nabla\conj{\psi}\otimes\nabla\psi)}+\frac{\lambda \sigma}{\sigma+1}\nabla|\psi|^{2\sigma+2}+\rho\nabla V=\frac{1}{4}\Delta\nabla\rho
+(i\Omega\cdot L)J-\Omega\wedge J.
\end{equation}

The proof of finite time bow-up now follows by the classical argument of Glassey (see \cite{caze}). To this end, we consider the time evolution of 
\begin{equation*}
I(t):=\frac{1}{2}\int_{\R^d}|x|^2|\psi(t, x)|^2 \, dx.
\end{equation*}
Differentiating with respect to time and using \eqref{eq:mass}, we obtain 
\begin{equation*}
\frac{d}{dt}I(t)=\int_{\R^d} x\cdot J(t, x)\, dx+\int_{\R^d}\frac{|x|^2}{2}(i\Omega\cdot L)\rho(t, x) \, dx.
\end{equation*}
Integrating by parts and using $(\Omega\cdot L) |x|^2 =0$ shows that the second integral in fact vanishes, i.e. we have
\begin{equation*}%\label{eq:first_der}
\frac{d}{dt}I(t)=\int_{\R^d} x\cdot J\;dx.
\end{equation*}
Differentiating in time once more and using \eqref{eq:current}, we obtain
\begin{align*}
\frac{d}{dt}\int x\cdot J\, dx=&\int x\cdot\bigg(-\diver\pare{\re(\nabla\conj{\psi}\otimes\nabla\psi)}
-\lambda\frac{\sigma}{\sigma+1}\nabla|\psi|^{2\sigma+2}-\rho\nabla V+\frac{1}{4}\Delta\nabla\rho\notag\\
&+(i\Omega\cdot L)J-\Omega\wedge J\bigg)dx,
\end{align*}
which we rewrite as
\begin{equation}
\begin{aligned}
\frac{d}{dt}\int x\cdot J\, dx=&\int|\nabla\psi|^2+\lambda\frac{d\sigma}{\sigma+1}|\psi|^{2\sigma+2}-\rho x\cdot\nabla V+x\cdot(i\Omega\cdot L)J\\
&-x\cdot\Omega\wedge J \;dx.\label{eq:45}
\end{aligned}
\end{equation}
Now we first note that for any potential $V(x)$ of the form \eqref{eq:Vquadr} we have $x\cdot\nabla V=2V$. Moreover, we compute
\begin{align*}
 \int_{\R^d} x\cdot (i\Omega\cdot L J) \; dx =& - \int_{\R^d} \left(\Omega\cdot L x \right) \cdot J \;dx = - \int_{\R^d} \left(\Omega\cdot(x\wedge\nabla) x \right) \cdot J \;dx\\
=& - \int_{\R^d} \left((\Omega\wedge x)\cdot\nabla\right) x \cdot J \;dx = - \int_{\R^d} (\Omega\wedge x)\cdot J \;dx, 
\end{align*}
shows that the last two terms in \eqref{eq:45} cancel each other. 
In summary we arrive at the following identity
\begin{equation} \label{eq:virial}
\frac{d^2}{dt^2}I(t)=\int|\nabla\psi|^2+\lambda\frac{d\sigma}{\sigma+1}|\psi|^{2\sigma+2}-2V|\psi|^2 \, dx,
\end{equation}
which is in fact exactly the same as in the case of NLS without rotation, c.f.\ \cite{caze}.

We can now prove assertion (i): Recall from Corollary \ref{cor:tilde_en} that if the potential $V(x)$ is axially symmetric, then $E_0(t)=E_0(0)$, with $E_0$ defined in \eqref{eq:tilde_en}. Hence from \eqref{eq:virial} and $V\geq 0$ we can write 
\begin{equation*}
\frac{d^2}{dt^2}I(t)\le2E_0+\lambda\frac{d\sigma-2}{\sigma+1}\int_{\R^d}|\psi|^{2\sigma+2}dx
\end{equation*}
Assuming $E_0<0$, $\lambda <0$, and $\sigma\ge2/d$, we consequently obtain
\begin{equation*}
\frac{d^2}{dt^2}I(t)<-C,
\end{equation*}
for some constant $C > 0$. Integrating this relation twice, we obtain
\begin{equation*}
I(t)<-\frac{C}{2}t^2+c_1t+c_2
\end{equation*}
with some integration constants $c_1$ and $c_2$. Thus, if the solution $\psi(t)\in \Sigma$ were to exist for all times, there would be a time $T^\ast < +\infty$, such that $I(T^\ast) < 0$. 
This however is in contradiction with the fact that, by definition, $I(t) \ge 0$ for all $t\in \R$ and hence the assertion is proved.

In order to prove assertion (ii) we again consider \eqref{eq:virial}: The problem is that in the case of a non-axially symmetric potential ($\Omega \cdot LV(x) \not =0$), 
the energy $E_0$ is no longer conserved. Rather we only have the conservation law for $E_\Omega(t) = E_\Omega(0)$. In order to use this 
piece of information, we first add and subtract to \eqref{eq:virial} a multiple of the angular momentum $L_\Omega(t)$, i.e.
\begin{equation*}
\frac{d^2}{dt^2}I(t)=\int_{\R^d}|\nabla\psi|^2+\frac{\lambda\sigma d}{\sigma+1}|\psi|^{2\sigma+2}-2V|\psi|^2-\alpha\conj{\psi}\Omega\cdot L\psi \, dx
+\int\alpha\conj{\psi}\Omega\cdot L\psi \, dx,
\end{equation*}
where $\alpha>0$ is a parameter to be chosen later on. 
Using Cauchy-Schwarz and Young's inequality, the last term on the r.h.s. can be bounded by
\begin{equation*}
\alpha\int_{\R^d}\conj{\psi}\Omega\cdot L\psi dx\le \alpha |\Omega|\|\nabla\psi\|_{L^2}\|x\psi\|_{L^2}
\le\frac{\alpha\theta}{2}\|\nabla\psi\|^2_{L^2}+\frac{\alpha |\Omega|^2}{2\theta}\|x\psi\|^2_{L^2},
\end{equation*}
where $\theta>0$ is another free parameter to be chosen later on. We consequently estimate
\begin{align*}
\frac{d^2}{dt^2}I(t)\le & \, \int_{\R^d}\pare{1+\frac{\alpha\theta}{2}}|\nabla\psi|^2+\frac{\lambda\sigma d}{\sigma+1}|\psi|^{2\sigma+2}
+\pare{-2V+\frac{\alpha |\Omega|^2}{2\theta}|x|^2}|\psi|^2 \, dx \\ & \,  - \int_{\R^d} \alpha\conj{\psi}\Omega\cdot L\psi \, dx.
\end{align*}
Now, we choose $\theta$ such that $2(1+\frac{\alpha\theta}{2})=\alpha$, that is $\theta=\frac{\alpha-2}{\alpha}$. In this way we have
\begin{multline*}
\frac{d^2}{dt^2}I(t)\le\int_{\R^d}\alpha\pare{\frac{1}{2}|\nabla\psi|^2+\lambda\frac{1}{\sigma+1}|\psi|^{2\sigma+2}+V|\psi|^2-\conj{\psi}\Omega\cdot L\psi}dx\\
+\int_{\R^d}\lambda\frac{\sigma d-\alpha}{\sigma+1}|\psi|^{2\sigma+2}dx+\int_{\R^d}\pare{-(\alpha+2)V+\frac{\alpha^2|\Omega|^2}{2(\alpha-2)}|x|^2}|\psi|^2dx.
\end{multline*}
Let $\underline \gamma:=\min(\gamma_1, \gamma_2, \gamma_3)$, and choose $\alpha$ such that
\begin{equation*}
\frac{(\alpha+2)}{2} \underline \gamma^2=\alpha^2\frac{|\Omega|^2}{2(\alpha-2)}.
\end{equation*}
This yields $\alpha = \alpha_\Omega$ with $$\alpha_\Omega=\sqrt{\frac{4\underline{\gamma}^2}{\underline{\gamma}^2-|\Omega|^2}} .$$ 
By doing so, the last term in the previous inequality is seen to be non-positive and furthermore we conclude that, for $\lambda <0$ and $\sigma\ge\frac{\alpha_\Omega}{d}$, it holds:
\begin{equation}
\frac{d^2}{dt^2}I(t)\le \alpha_\Omega E_\Omega(t) \equiv  \alpha_\Omega E_\Omega(0).
\end{equation}
Thus, if the initial energy $E_\Omega(0)<0$ the second derivative of $I(t)$ is again negative and we can argue (by contradiction) as before.
\end{proof}

\section{Concluding remarks}
\label{sec:final}

As we have seen above, equation \eqref{eq:NLS_rot} can be considered (upon a change of coordinates) as a special case of NLS with time-dependent potentials (sub-quadratic in $x$). 
This class of models has recently been studied in \cite{Car09}. Following the arguments given therein, 
one could infer global in-time existence of \eqref{eq:NLS_rot} 
for \emph{sufficiently small} initial data $ \psi_0 \in \Sigma$, regardless of the sign of the nonlinearity. 
Moreover, growth rates for higher order (weighted) Sobolev norms can also be obtained as in \cite{Car09}.
In addition, we note 
that for a \emph{repulsive}, isotropic quadratic potential $V(x)=-\frac{\gamma^2}{2} |x|^2$, 
the time-dependent change of coordinates is trivial and we could henceforth conclude global in-time existence for sufficiently large $\gamma>0$ by following 
the arguments given in \cite{Car03}. 

We also want to point out that for the usual NLS with $\sigma = 2/d$ there is an extra symmetry which has been successfully deployed in the study of blow-up (yielding explicit blow-up solutions and blow-up rates), see e.g. \cite{Ra}. 
Using the so-called \emph{Lens transform} \cite{KaWe} one can transfer (most of) these results to the case of 
NLS with isotropic time dependent quadratic potential $W(t,x) = \gamma(t) |x|^2$, see \cite{Car09}. However, it is argued in \cite{Car09} that such an approach is only feasible in the case of isotropic potentials and thus, 
we cannot expect from it any further insight on the possibility of blow-up in our case, when $(L\cdot \Omega) V(x)\not =0$ and $|\Omega|> \underline \gamma$.

Finally, it is worth noting that the effect of the angular momentum rotation term in our model is very different from other situations. For example, 
it has been shown for the Euler equations with Coriolis force that blow-up can be delayed through a sufficiently strong rotation term \cite{liuTad} (see also as \cite{kdvRot} for a related result). 
Clearly, the situation in our model is much more involved, and we can not expect an analogous result to be true (the counterexample being the case where $V(x)$ is axially symmetric).

\end{document}